\documentclass[reqno]{amsart}
\usepackage{hyperref}
\usepackage{color}

\def\N{{\mathbb{N}}}

\def\R{{\mathbb{R}}}

\def\S{{\mathbb{S}}}

\begin{document}
\title[Proof of Pontryagin principle]{On the proof of Michel of the Maximum Pontryagin principle}
\author[J. Blot, H. Yilmaz]{ Jo${\rm \ddot e}$l Blot \& Hasan Yilmaz}

\address{Jo\"{e}l Blot: Laboratoire SAMM EA 4543,\newline
Universit\'{e} Paris 1 Panth\'{e}on-Sorbonne, centre P.M.F.,\newline
90 rue de Tolbiac, 75634 Paris cedex 13,
France.}
\email{blot@univ-paris1.fr}
\address{Hasan Yilmaz: Laboratoire LPSM UMR 8001 \newline
Universit\'e Paris-Diderot, Sorbonne-Paris-Cit\'e \newline
b\^atiment Sophie Germain, 8 place Aur\'elie Nemours, \newline
75013 Paris, France.}
\email{yilmaz@lpsm.paris}
\date{April, 1, 2019}
\numberwithin{equation}{section}
\newtheorem{theorem}{Theorem}[section]
\newtheorem{lemma}[theorem]{Lemma}
\newtheorem{example}[theorem]{Example}
\newtheorem{remark}[theorem]{Remark}
\newtheorem{definition}[theorem]{Definition}
\newtheorem{corollary}[theorem]{Corollary}
\newtheorem{proposition}[theorem]{Proposition}
\thispagestyle{empty} \setcounter{page}{1}
\begin{abstract}
We provide an improvment of the maximum principle of Pontryagin of the optimal control problems, for a system governed by an ordinary differential equation, in presence of final constraints, in the setting of the piecewise differentiable state functions (valued in a Banach space) and of piecewise continuous control functions (valued in a metric space). As Michel we use the needlelike variations, but we introduce tools of functional analysis and a recent multiplier rule of the static optimization to make our proofs.
\end{abstract}
\maketitle
\vskip4mm
\section{Introduction}
The paper deals with the maximum principle of Pontryagin for a problem of Bolza in the following form.
\[
({\mathcal B})
\left\{
\begin{array}{cl}
{\rm Maximize} & \int_0^T f^0(t,x(t),u(t))dt + g^0(x(T)) \\
{\rm subject \;  to} & x \in PC^1([0,T], \Omega), u \in PC^0([0,T], U)\\
\null & x'(t) = f(t,x(t), u(t)), \; x(0) = \xi_0\\
\null & \forall \alpha = 1,...,m, \; \; g^{\alpha}(x(T)) \geq 0\\
\null & \forall \beta = 1,..., q, \; \; h^{\beta}(x(T)) = 0.
\end{array}\right.
\]
In the special case where $f^0$ is equal to zero, the problem is called a problem of Mayer and it is denoted by (${\mathcal M}$). $T \in (0, + \infty)$ is fixed. $E$ denotes a real Banach space, $\Omega$ is a nonempty open subset of $E$, $U$ denotes a nonempty metric space, and $\xi_0 \in \Omega$ is fixed; we use the mappings $f^0 : [0,T] \times \Omega \times U \rightarrow \R$ and $f : [0,T] \times \Omega \times U \rightarrow E$. The real valued functions $g^{\alpha}$ and $h^{\beta}$ are defined on $\Omega$, and $m$ and $q$ are fixed integer numbers.\\
$PC^0([0,T], U)$ denotes the space of the piecewise continuous functions from $[0,T]$ into $U$, and $PC^1([0,T], \Omega)$ denotes the space of the piecewise differentiable functions from $[0,T]$ into $\Omega$. The precise definitions of these notions are given in Section 2.
\vskip2mm
\noindent
When $(x,u)$ is an admissible process for $(\mathcal{B})$ or $(\mathcal{M})$, we consider the following condition of qualification, $i \in \{0,1 \}$. (QC, 0) is due to Michel, \cite{Mi2}. 
\[
(QC, i)
\left\{
\begin{array}{l}
{\rm If} \;\; (c_{\alpha})_{i \leq \alpha \leq m} \in \R_+^{1 - i+ m}, (d_{\beta})_{1 \leq \beta \leq q} \in \R^q {\rm \;\;\;satisfy} \\
(\forall \alpha = 1,...,m, \; c_{\alpha} g^{\alpha}(x(T)) = 0), {\rm and} \\
\sum_{\alpha = i}^m c_{\alpha} D g^{\alpha}(x(T)) + \sum_{\beta = 1}^q d_{\beta} D h^{\beta}(x(T)) = 0, {\rm then}\\
(\forall \alpha = i, ...,m, \;  c_{\alpha} = 0) \;\; {\rm and} \;\; (\forall \beta = 1, ..., q, \; d_{\beta} = 0).
\end{array}
\right.
\]
\noindent
The main theorems of the paper are the following ones.
\vskip2mm
\begin{theorem}\label{th11} Let $(x_0,u_0)$ be a solution of problem (${\mathcal B}$). We assume that the following assumptions are fulfilled.
\begin{itemize}
\item[(A1)] For all $\alpha \in \{0,...,m\}$, $g^{\alpha}$ is Fr\'echet differentiable at $x_0(T)$.
\item[(A2)] For all $\beta \in \{1,...,q \}$, $h^{\beta}$ is continuous on a neighborhood of $x_0(T)$ and is Fr\'echet differentiable at $x_0(T)$.
\item[(A3)] $f^0$ is continuous on $[0,T] \times \Omega \times U$, the partial differential with respect to the second vector variable $D_2f^0(t,\xi, \zeta)$ exists for all $(t, \xi, \zeta) \in [0,T] \times \Omega \times U$, and $D_2f^0$ is continuous on $[0,T] \times \Omega \times U$.
\item[(A4)] $f$ is continuous on $[0,T] \times \Omega \times U$, the partial differential with respect to the second vector variable $D_2f(t,\xi, \zeta)$ exists for all $(t, \xi, \zeta) \in [0,T] \times \Omega \times U$, and $D_2f$ is continuous on $[0,T] \times \Omega \times U$.
\end{itemize}
Then there exists $(\lambda_{\alpha})_{0 \leq \alpha \leq m} \in \R^{1 + m}$, $(\mu_{\beta})_{1 \leq \beta \leq q} \in \R^q$ and $p \in PC^1([0,T], E^*)$ which satisfy the following conditions.\\
{\bf Part (I)}
\begin{itemize}
\item[(NN)] $(\lambda_{\alpha})_{0 \leq \alpha \leq m}$ and $(\mu_{\beta})_{1 \leq \beta \leq q}$ are not simultaneously equal to zero.
\item[(Si)] For all $\alpha \in \{0,...,m \}$, $\lambda_{\alpha} \geq 0$.
\item[(S${\ell}$)] For all $\alpha \in \{1,...,m \}$, $\lambda_{\alpha} g^{\alpha}(x_0(T)) = 0$.
\item[(TC)] $\sum_{\alpha = 0}^{m} \lambda_{\alpha} Dg^{\alpha}(x_0(T)) + \sum_{\beta = 1}^q \mu_{\beta} D h^{\beta}(x_0(T)) = p(T)$.
\item[(AE.B)] $p'(t) = - D_2 H_B(t, x_0(t),u_0(t), p(t), \lambda_0)$ for all $t \in [0,T]$ except at most when $t$ is a discontinuity point of $u_0$.
\item[(MP.B)] For all $t \in [0,T]$, for all $\zeta \in U$, \\
$H_B(t, x_0(t), u_0(t), p(t), \lambda_0) \geq H_B(t, x_0(t), \zeta, p(t), \lambda_0)$.
\item[(CH.B)] $\bar{H}_B := [ t \mapsto H_B(t, x_0(t), u_0(t), p(t), \lambda_0)] \in C^0([0,T], \R)$.
\end{itemize}
{\bf Part (II)} If in addition we assume that, for all $(t,\xi, \zeta) \in [0,T] \times \Omega \times U$, the partial derivatives with respect to the first variable $\partial_1f^0(t, \xi, \zeta)$ and $\partial_1f(t, \xi, \zeta)$ exist and $\partial_1f^0$ and $\partial_1f$ are continuous on $[0,T] \times \Omega \times U$, then $\bar{H}_B \in PC^1([0,T], \R)$ and, for all $t \in [0,T]$ which is a continuity point of $u_0$, $\bar{H}_B'(t) = \partial_1 H_B(t, x_0(t), u_0(t), p(t), \lambda_0)$.\\
{\bf Part (III)} If we assume that (QC, 1) is fulfilled for $(x,u) = (x_0,u_0)$ then, for all $t \in [0,T]$, $(\lambda_0, p(t))$ is never equal to zero.
\end{theorem}
In this statement, $E^*$ denotes the topological dual space of $E$, (NN) is a condition of non nullity, (Si) is a sign condition, (S${\ell}$) is a slackness condition, (TC) is the transversality condition, (AE.B) is the adjoint equation where the Hamiltonian of the problem of Bolza is defined as $H_B(t,x,u,p, \lambda) := \lambda f^0(t,x,u) + p \cdot f(t,x,u)$. (MP.B) is the maximum principle and (CH.B) is a condition of continuity on the Hamiltonian.
\begin{theorem}\label{th12} Let $(x_0,u_0)$ be a solution of $(\mathcal{M})$. Under (A1), (A2), and (A4) there exist $(\lambda_{\alpha})_{0 \leq \alpha \leq m} \in \R^{1+m}$, $(\mu_{\beta})_{1 \leq \beta \leq q} \in \R^q$ and $p \in PC^1([0,T], E^*)$ such that the following conditions hold.\\
{\bf Part (I)}
\begin{itemize}
\item[(NN)] $(\lambda_{\alpha})_{0 \leq \alpha \leq m}$ and $(\mu_{\beta})_{1 \leq \beta \leq q}$ are not simultaneously equal to zero.
\item[(Si)] For all $\alpha \in \{0,...,m \}$, $\lambda_{\alpha} \geq 0$.
\item[(S${\ell}$)] For all $\alpha \in \{1,...,m \}$, $\lambda_{\alpha} g^{\alpha}(x_0(T)) = 0$.
\item[(TC)] $\sum_{\alpha = 0}^{m} \lambda_{\alpha} Dg^{\alpha}(x_0(T)) + \sum_{\beta = 1}^q \mu_{\beta} D h^{\beta}( x_0(T)) = p(T)$.
\item[(AE.M)] $p'(t) = - D_2 H_M(t, x_0(t),u_0(t), p(t))$ for all $t \in [0,T]$ except at most when $t$ is a discontinuity point of $u_0$.
\item[(MP.M)] For all $t \in [0,T]$, for all $\zeta \in U$, \\
$H_M(t, x_0(t), u_0(t), p(t)) \geq H_M(t, x_0(t), \zeta, p(t))$.
\item[(CH.M)] $\bar{H}_M := [ t \mapsto H_M(t, x_0(t), u_0(t), p(t))] \in C^0([0,T], \R)$.
\end{itemize}
{\bf Part (II)} If we assume that, for all $(t,\xi, \zeta) \in [0,T] \times \Omega \times U$, the partial derivative with respect to the first variable $\partial_1f(t, \xi, \zeta)$ exists and $\partial_1f$ is continuous on $[0,T] \times \Omega \times U$, then we have $\bar{H}_M \in PC^1([0,T], \R)$ and, for all $t \in [0,T]$ which is a continuity point of $u_0$, $ \bar{H}_M'(t) = \partial_1 H_M(t, x_0(t), u_0(t), p(t))$.\\
{\bf Part (III)} If in addition of (A1), (A2), (A4), we assume that (QC, 0) is fulfilled when $(x,u) = (x_0,u_0)$, then $p(t)$ is never equal to zero when $t \in [0,T]$.
\end{theorem}
In this statement the Hamiltonian of the problem of Mayer is defined as\\
 $H_M(t,x,u,p) := p f(t,x,u)$.
\vskip2mm
\noindent
To prove these statements, we build a variation of the proof of Michel \cite{Mi2} (on the problem of Mayer) by introducing functional analytic arguments. Notably we consider special function spaces of piecewise continuous functions, operators on these function spaces and  fixed point theorems. We also use a recent result of multiplier rule on  static optimization problems. The main contributions of the paper are the following ones.
\begin{itemize}
\item Our assumptions on the $g^{\alpha}$ are only their Fr\'echet diffferentiability, and on the $h^{\beta}$ are their continuity and their Fr\'echet differentiability, not their continuous differentiability as in \cite{Mi2}, \cite{ATF} (p. 321)  and \cite{IT} (p. 132).
\item In \cite{ATF}( p. 321) and \cite{IT} (p. 132) the first conclusion of the theorem of Pontryagin is that $(\lambda_{\alpha})_{0 \leq \alpha \leq m}$, $(\mu_{\beta})_{1 \leq \beta \leq q}$, and $p$ are not simultaneously equal to zero. In our Theorem \ref{th11}, the first conclusion is that $(\lambda_{\alpha})_{0 \leq \alpha \leq m}$ and $(\mu_{\beta})_{1 \leq \beta \leq q}$ are not simultaneously equal to zero; it is an improvment.
\item As in \cite{Mi2} we do not demand the finiteness of the dimension of the space $E$; in \cite{ATF} and in \cite{IT}, $E$ is finite-dimensional. Moreover we use an open subset of $E$ instead of $E$; it is another difference with \cite{Mi2}. Ever about \cite{Mi2}, we prove a condition of continuity of the needlike variations with respect to the thickness of the needles, which is useful but omitted in \cite{Mi2}.
\end{itemize}
Note that  there exist statements of theorem of Pontryagin without assumptions of continuous differentiablity, by using locally Lipschitzean mappings and generalized differential calculus on these mappings, e.g. in \cite{CLSW}. A mapping which is Fr\'echet differentiable at a point is not necessarily locally Lipschitzean, and conversely a mapping which is locally Lipschitzean is not necessarily Fr\'echet  differentiable at a given point; hence our result is not comparable with the statements of the locally Lipschitzean setting.
\section{Function spaces}
When $X$ and $Y$ are metric spaces, $C^0(X,Y)$ denotes the space of the continuous mappings from $X$ into $Y$. When $X$ is an open subset of a real normed vector space or an interval of $\R$, $C^1(X,Y)$ denotes the space of the continuously Fr\'echet differentiable mappings from $X$ into $Y$. When $X$ and $Y$ are real normed vector spaces, ${\mathcal L}(X,Y)$ denotes the space of the bounded linear mappings from $X$ into $Y$, and $Isom(X,X)$ denotes the space of the topological isomorphisms from $X$ onto $X$. When $X$ is a metric space, $x \in X$ and $r \in \R_{+*} := (0, + \infty)$, the closed ball (respectively open ball) centered at $x$ with a radius equal to $r$ is denoted by $\overline{B}(x,r)$ (respectively $B(x,r)$).
\subsection{Piecewise continuous functions.}
Let $Y$ be a metric space. A function $u : [0,T] \rightarrow Y$ is called piecewise continuous when $u \in C^0([0,T], Y)$ or when there exists a subdivision $0= \tau_0 < \tau_1 < ...< \tau_k < \tau_{k+1} = T$ such that
\begin{itemize}
\item For all $i \in \{0,...,k \}$, $u$ is continuous on $(\tau_i, \tau_{i+1})$.
\item For all $i \in \{0,...,k \}$, the right-hand limit $u(\tau_i +)$ exists in $Y$.
\item For all $i \in \{1,...,k+1 \}$, the left-hand limit $u(\tau_i -)$ exists in $Y$.
\end{itemize}
In other words, such a function is a regulated function (cf. \cite{Bo}, chapter 2 ) which possesses at most a finite  number of discontinuity points. Their space is denoted by $PC^0([0,T], Y)$. $PC^0([0,T], Y, (\tau_i)_{0 \leq i \leq k+1})$ denotes the space of the $u \in PC^0([0,T], Y)$ such that the set of the discontinuity points of $u$ is included in $\{ \tau_i : i \in \{0,...,k+1 \} \}$. When $A$ is a subset of $Y$, $PC^0([0,T], A)$ (respectively $PC^0([0,T], A, (\tau_i)_{0 \leq i \leq k+1})$) denotes the space of the $u \in PC^0([0,T], Y)$ (respectively $PC^0([0,T], A, (\tau_i)_{0 \leq i \leq k+1})$) such that the closure $\overline{u([0,T])} \subset A$.
\begin{definition}\label{def211} A function $u \in PC^0([0,T], A)$ is called a normalized piecewise continuous function when moreover $u$ is right continuous on $[0,T)$ and when $u(T -) = u(T)$.
\end{definition}
The space of such functions is denoted by $NPC^0([0,T],A)$. When $(\tau_i)_{0 \leq i \leq k+1}$ is a subdivision of $[0,T]$, we set
$$NPC^0([0,T],A, (\tau_i)_{0 \leq i \leq k+1}) :=  NPC^0([0,T],A) \cap PC^0([0,T], A, (\tau_i)_{0 \leq i \leq k+1}).$$
\subsection{Piecewise continuously differentiable functions.}
When $E$ is a real Banach space, a function $x : [0,T] \rightarrow E$ is called piecewise continuously differentiable when $x \in C^0([0,T], E)$ and when $x \in C^1([0,T], E)$ or when there exists a subdivision $(\tau_i)_{0 \leq i \leq k+1}$ of $[0,T]$ such that the following conditions are fulfilled.
\begin {itemize}
\item For all $i \in \{0,...,k \}$, $x$ is $C^1$ on $(\tau_i, \tau_{i+1})$
\item For all $i \in \{0,...,k \}$, $x'(\tau_i +)$ exists in $E$
\item For all $i \in \{1,...,k+1 \}$, $x'(\tau_i -)$ exists in $E$.
\end{itemize}
The $\tau_i$ are the corners of the function $x$.
We denote by $PC^1([0,T], E)$ the space of such functions; this space is denoted by $KC^1([0,T],E )$ in \cite{ATF} (p. 66, Section 1.4). When $\Omega$ is an open subset of $E$, $PC^1([0,T], \Omega)$ is the set of the $x \in PC^1([0,T], E)$ such that $x([0,T]) \subset \Omega$. When  $(\tau_i)_{0 \leq i \leq k+1}$ is a subdivision of $[0,T]$, we denote by $PC^1([0,T], E, (\tau_i)_{0 \leq i \leq k+1})$ the set of the $x \in PC^1([0,T], E)$ such that the set of the corners of $x$ is included in $\{ \tau_i : i \in \{0, ..., k+1 \} \}$.\\
When $x \in PC^1([0,T], E, (\tau_i)_{0 \leq i \leq k+1})$, we define the function $\underline{d}x : [0,T] \rightarrow E$ by setting
\begin{equation}\label{eq21}
\underline{d}x(t) :=
\left\{
\begin{array}{ccl}
x'(t) & {\rm if} & t \in [0,T] \setminus \{ \tau_i : i \in \{0, ..., k+1 \}\}\\
x'(\tau_i +) & {\rm if} & t = \tau_i, i \in \{0,...,k \}\\
x'(T-) & {\rm if} & t = T.
\end{array}
\right.
\end{equation}
Note that $\underline{d}x \in NPC^0([0,T], E, (\tau_i)_{0 \leq i \leq k+1})$.
%
\subsection{Rewording of the problems}
We consider the following problem.
\[
({\mathcal B'})
\left\{
\begin{array}{cl}
{\rm Maximize} & J(x,u) := \int_0^T f^0(t,x(t),u(t))dt + g^0(x(T)) \\
{\rm subject \;  to} & x \in PC^1([0,T], \Omega), u \in NPC^0([0,T], U)\\
\null & \underline{d}x(t) = f(t,x(t), u(t)), \; x(0) = \xi_0\\
\null & \forall \alpha = 1,...,m, \; \; g^{\alpha}(x(T)) \geq 0\\
\null & \forall \beta = 1,..., q, \; \; h^{\beta}(x(T)) = 0.
\end{array}
\right.
\]
We denote by (${\mathcal M}'$) the special case of (${\mathcal B}'$) where $f^0 = 0$.
\vskip2mm
\noindent
We denote by $Adm({\mathcal B})$ (respectively $Adm({\mathcal B}')$) the set of the admissible processes of $({\mathcal B})$ (respectively $({\mathcal B}')$). When $(x,u) \in Adm({\mathcal B})$, and when the discontinuity points of $u$ are in the values of the subdivision $(\tau_i)_{0 \leq i \leq k+1}$ of $[0,T]$, we introduce the fonction 
\begin{equation}\label{eq22}
\overline{u}(t) :=
\left\{
\begin{array}{ccl}
u(t) & {\rm if} & t \in (\tau_i, \tau_{i+1}), i \in \{0,...,k \}\\\
u(\tau_i +) & {\rm if} & t = \tau_i, i \in \{0, ...,k \} \\
u(T -) & {\rm if} & t = T;
\end{array}
\right.
\end{equation}
we have $\overline{u} \in NPC^0([0,T], U)$.\\
Note that $f(t, x(t), u(t))$ and $f(t,x(t), \overline{u}(t))$ can to be diffferent only when $t \in \{ \tau_i : 0 \leq i \leq k +1\}$ and so we have $\underline{d}x(t) = f(t, x(t), \overline{u}(t))$ for all $t \in [0,T]$. Also note that 
$f^0(t, x(t), u(t))$ and $f^0(t,x(t), \overline{u}(t))$ can to be diffferent only when $t \in \{ \tau_i : 0 \leq i \leq k+1 \}$ and so we have $\int_0^T f^0(t,x(t), u(t))dt = \int_0^T f^0(t,x(t), \overline{u}(t)) dt$. Consequently we obtain $J(Adm({\mathcal B})) = J(Adm({\mathcal B}'))$. When $(x_0,u_0)$ is a solution of $({\mathcal B}')$ then it is also a solution of $({\mathcal B})$. Conversely, when $(x_0,u_0)$ is a solution of $({\mathcal B})$, building $\overline{u_0}$ by using (\ref{eq22}) where $u$ is $u^0$, we obtain that $(x_0, \overline{u_0})$ is a solution of $({\mathcal B}')$. It is why we can say that the problems $({\mathcal B})$ and $({\mathcal B}')$ are equivalent problems. A similar reasoning is valid to show that the problems $(\mathcal{M})$ and $(\mathcal{M}')$ are equivalent.
 %
\section{The needlelike variations}
\subsection{Two results of the metric spaces theory}
The first result is a generalization of the theorem of Heine on the uniform continuity of a continuous mapping on a compact metric space; it is useful to avoid an assumption of local compactness, and specially, in normed vector spaces, to avoid an assumption of finiteness of the dimension.
\begin{theorem}\label{th31} (\cite{Sc} p. 355, note (**))
Let $X$ and $Y$ be two metric spaces, $\phi \in C^0(X,Y)$, and $K \subset X$ be a compact. Then we have 
$$\forall \epsilon > 0, \exists \delta_{\epsilon} > 0, \forall x \in X, \forall z \in K, d(x,z) \leq \delta_{\epsilon} \Longrightarrow d(\phi(x), \phi(z)) \leq \epsilon.$$
\end{theorem}
The following result is a theorem of fixed points in presence of parameters.
\begin{theorem}\label{th32}(\cite{Sc} p. 103, Theorem 46-bis )
Let $X$ be a complete metric space, $\Lambda$ be a metric space, and $\phi : X \times \Lambda \rightarrow X$ be a mapping. We assume that the following conditions are fulfilled.
\begin{itemize}
\item[(a)] $\forall x \in X, \phi(x, \cdot) \in C^0(\Lambda, X).$
\item[(b)] $\exists k \in [0,1), \forall \lambda \in \Lambda, \forall x, z \in X, d(\phi(x, \lambda), \phi(z, \lambda)) \leq k d(x,z).$
\end{itemize}
Then we have
\begin{itemize}
\item[(i)] $\forall \lambda \in \Lambda, \exists ! x_{\lambda} \in X, \phi(x_{\lambda}, \lambda) = x_{\lambda}.$
\item[(ii)] $[ \lambda \mapsto x_{\lambda}] \in C^0(\Lambda, X)$.
\end{itemize}
\end{theorem}
\subsection{Definitions of the needlelike variations}
We follow the definition of Michel of the needlelike variations which is given in \cite{Mi2}; Michel himself refers to \cite{PBGM} for this approach. Let $(x_0,u_0)$ be a solution of $({\mathcal M}')$.
 When $N \in \N_* := \N \setminus \{0 \}$, we consider $S := ((t_i,v_i))_{1 \leq i \leq N}$ where $t_i \in [0,T]$ satisfying $0 < t_1 \leq t_2 \leq ... \leq t_N < T$, and where $v_i \in U$. We denote by $\S$ the set of such $S$.
 \vskip1mm
 \noindent 
When $S \in \S$ and $a = (a_1, ..., a_N) \in \R^N_+$, we define the following objects
$$J(i) = J(i, S) := \{ j \in \{ 1,...,i-1 \} : t_j = t_i \}$$
$$b_i(a) = b_i(a,S) := 
\left\{
\begin{array}{ccl}
0 & {\rm if} & J(i) = \emptyset\\
\sum_{j \in J(i)} a_j & {\rm if} & J(i) \neq \emptyset.
\end{array}
\right.
$$
$$I_i(a) = I_i(a,S) := [ t_i + b_i(a), t_i+  b_i(a) + a_i ),$$
\begin{equation}\label{eq31}
u_a(t) = u_a(t,S) := 
\left\{
\begin{array}{ccl}
v_i & {\rm if} & t \in I_i(a), 1 \leq i \leq N\\
u_0(t) & {\rm if} & t \in [0,T] \setminus \cup_{1 \leq i \leq N} I_i(a).
\end{array}
\right.
\end{equation}
When $a$ is small enough, we have $I_i(a) \subset [0,T]$ and $I_i(a) \cap I_j(a) = \emptyset$ when $i \neq j$. \\
We will prove the existence of a solution, denoted by $x_a$ (which depends on $S$ and $a$) of the following Cauchy problem on $[0,T]$:
\begin{equation}\label{eq32}
\underline{d}x_a(t) = f(t,x_a(t), u_a(t)),\; \;  x_a(0) = \xi_0.
\end{equation}
In the sequel of this section, we arbitrarily fix a $S = (t_i, v_i)_{1 \leq i \leq N}$ in $\S$.
\subsection{Properties of continuity} In this subsection, we establish the existence of $x_a$ on $[0,T]$ all over and we establish the continuity of the mapping $[a \mapsto x_a]$. To do that we introduce an appropriate function space and a nonlinear operator from which $x_a$ appears as a fixed point of this operator. The continuity of $[a \mapsto x_a]$ will be a consequence of the fixed point theorem with parameters.
\begin{lemma}\label{lem33} (\cite{Mi2}, Proposition 2) There exists $k \in \R_{+*} := \R_+ \setminus \{0 \}$, there exists $\rho \in \R_{+*}$ such that, for all $a \in \R^N_+$ satisfying $\Vert a \Vert \leq \rho$, we have 
$$\int_0^T \Vert f(t, x_0(t), u_a(t)) - f(t, x_0(t), u_0(t)) \Vert dt \leq k \Vert a \Vert.$$
\end{lemma}
\noindent
We consider the subdivision $(\tau_i)_{0 \leq i \leq k+1}$ of $[0,T]$ where the $\tau_i$ are the discontinuity points of $u_0$. For all $i \in \{ 0,...k \}$ we consider the function $u_0^i : [ \tau_i, \tau_{i + 1}] \rightarrow U$ defined by
\begin{equation}\label{eq33}
u_0^i(t) :=
\left\{
\begin{array}{ccl}
u_0(t) & {\rm if} & t \in [ \tau_i, \tau_{i+1} ) \\
u_0(\tau_{i+1} -) & {\rm if} & t = \tau_{i+1}.
\end{array} 
\right.
\end{equation}
Hence $u_0^i \in C^0([ \tau_i, \tau_{i+1} ], U)$, and consequently $u_0^i([ \tau_i, \tau_{i+1}])$ is compact. We set
\begin{equation}\label{eq34}
M := (\bigcup_{0 \leq i \leq k} u_0^i([ \tau_i, \tau_{i+1}]) \cup \{ v_i : 1 \leq i \leq N \}).
\end{equation}
$M$ is compact as a finite union of compacts. We set
\begin{equation}\label{eq35}
\Gamma := \{ (t, x_0(t)) : t \in [0,T] \}.
\end{equation}
Since $x_0 \in C^0([0,T], \Omega)$, $\Gamma$ is compact.
\begin{lemma}\label{lem34} There exist $L \in \R_{+*}$ and $r \in \R_{+*}$  such that, $\forall t \in [0,T]$,\\
 $\forall \xi, \xi_1 \in \overline{B}(x_0(t), r)$, $\forall \zeta \in M$, we have 
$\Vert f(t, \xi, \zeta) - f(t, \xi_1, \zeta) \Vert \leq L \Vert \xi - \xi_1 \Vert.$
\end{lemma}
\begin{proof}
Since $\Omega$ is open in $E$, since $x_0([0,T])$ is compact and included in $\Omega$, there exists $\gamma > 0$ such that $\{ \xi \in E : d( \xi, x_0([0,T])) < \gamma \} \subset \Omega$, where $d(\xi, x_0([0,T]) := \inf_{0 \leq t \leq T} \Vert  \xi - x_0(t) \Vert$.
We set $K := \Gamma \times M$; $K$ is compact as a product of compacts. Using Theorem \ref{th31} and (A4), we have
\begin{equation}\label{eq36}
\left.
\begin{array}{r}
\forall \epsilon > 0, \exists \delta_{\epsilon} \in (0, \gamma), \forall (t,\xi, \zeta) \in K, \forall (t_1, \xi_1, \zeta_1) \in [0,T] \times \Omega \times U,\\
d((t,\xi, \zeta), (t_1, \xi_1, \zeta_1)) = \vert t - t_1 \vert + \Vert \xi - \xi_1 \Vert + d(\zeta, \zeta_1) \leq \delta_{\epsilon} \Longrightarrow \\
\Vert D_2 f(t,\xi, \zeta) - D_2 f(t_1, \xi_1, \zeta_1) \Vert \leq \epsilon.
\end{array} 
\right\}
\end{equation}
Arbitrarily fix an $\epsilon > 0$. Let $t \in [0,T]$, $\zeta \in M$ and $\xi \in \overline{B}(x_0(t), \delta_{\epsilon})$. From (\ref{eq36}) we obtain
$$
\begin{array}{rcl}
\Vert D_2f(t,\xi,\zeta) \Vert&  \leq&  \Vert D_2f(t,x_0(t), \zeta) \Vert + \epsilon \\
\null & \leq & \sup_{(t_1, \xi_1, \zeta_1) \in K} \Vert D_2 f(t_1, \xi_1, \zeta_1) \Vert + \epsilon \Longrightarrow
\end{array}
$$
$L := \sup \{ \Vert D_2 f(t, \xi_1, \zeta_1) \Vert : t \in [0,T], \xi_1 \in \overline{B}(x_0(t), \delta_{\epsilon}), \zeta_1 \in M \}$\\\
$\leq \sup_{(t_1, \xi_1, \zeta_1) \in K} \Vert D_2 f(t_1, \xi_1, \zeta_1) \Vert + \epsilon < + \infty$.\\
We set $r := \delta_{\epsilon}$. If $t \in [0,T]$, $\xi , \xi_1 \in \overline{B}(x_0(t), r)$ and $\zeta \in M$, using the Mean Value Inequality of the differential calculus theory we obtain \\
$\Vert f(t, \xi, \zeta) - f(t, \xi_1, \zeta) \Vert \leq L \Vert \xi - \xi_1 \Vert$.
\end{proof}
When $\varphi \in C^0([0,T], E)$, we set $\Vert \varphi \Vert_b := \sup_{t \in [0,T]} (e^{- L t} \Vert \varphi (t) \Vert)$. $\Vert \cdot \Vert_b$ is called the norm of Bielecki (\cite{DG} p. 25-27) and $(C^0([0,T], E), \Vert \cdot \Vert_b)$ is a complete normed vector space. We define 
\begin{equation}\label{eq37}
r_1 := r e^{-L \cdot T}, \;\; \mathcal{X} := \overline{B}(x_0, r_1).
\end{equation}
Note that $(\mathcal{X}, \Vert \cdot \Vert_b)$ is a complete metric space. Note that when $x \in \overline{B}(x_0, r_1)$ we have, for all $t \in [0,T]$, $e^{ -L \cdot T} \Vert x(t) - x_0(t) \Vert \leq e^{- L \cdot t} \Vert x(t) - x_0(t) \Vert \leq   e^{- L \cdot T} r$ which implies $\Vert x(t) - x_0(t) \Vert \leq   r < \gamma$, and so $x(t) \in \Omega$. For all $a \in \overline{B}(0, \rho) \cap \R^N_+$, we consider the operator
\begin{equation}\label{eq38}
\Phi_a : \mathcal{X} \rightarrow C^0([0,T], E), \; \Phi_a(x) := [t \mapsto \xi_0 + \int_0^t f(s, x(s), u_a(s)) ds].
\end{equation}
\begin{lemma}\label{lem35} The constants $k$ and $\rho$ come from Lemma \ref{lem33}; the constant $L$ comes from Lemma \ref{lem34} and the constant $r_1$ comes from (\ref{eq37}).\\
 We set $r_2 := \min \{ \rho, e^{-L \cdot T} r_1 k^{-1}\}$.
When $a \in \R^N_+$, if $\Vert a \Vert \leq r_2$ then $\Phi_a(\mathcal{X}) \subset \mathcal{X}$.
\end{lemma}
\begin{proof}
Note that, for all $t \in [0,T]$, we have $x_0(t) = \xi_0 + \int_0^t f(s, x_0(s), u_0(s)) ds$; consequently we have
$\Vert \Phi_a(x_0)(t) - x_0(t) \Vert = \Vert \int_0^t (f(s,x_0(s), u_a(s)) - f(s, x_0(s), u_0(s))) ds \Vert$ \\
$\leq \int_0^t \Vert f(s,x_0(s), u_a(s)) - f(s,x_0(s), u_0(s)) \Vert ds \;\;\;\; \Longrightarrow $  \\
$e^{-L \cdot t} \Vert \Phi_a(x_0)(t) - x_0(t) \Vert  \leq e^{- L \cdot t} \int_0^t \Vert f(s,x_0(s), u_a(s)) - f(s,x_0(s), u_0(s)) \Vert ds$ \\
$ \leq e^{- L \cdot t} \int_0^T \Vert f(s,x_0(s), u_a(s)) - f(s,x_0(s), u_0(s)) \Vert ds \leq e^{- L \cdot t} k \Vert a \Vert $ using Lemma \ref{lem33}. Hence taking the sup on the $t \in [0,T]$, we obtain \\
$\Vert \Phi_a(x_0) - x_0 \Vert_b \leq \sup_{t \in [0,T]}e^{- L \cdot t} k \Vert a \Vert \leq k \Vert a \Vert$, and so we have, for $a \in \R^N_+$, 
\begin{equation}\label{eq39}
\Vert \Phi_a(x_0) - x_0 \Vert_b \leq e^{-L \cdot T} r_1.
\end{equation}
Let $x \in \mathcal{X}$; then for all $t \in [0,T]$, we have $e^{-L \cdot t} \Vert x(t) - x_0(t) \Vert \leq r_1$ which implies $\Vert x(t) - x_0(t) \Vert \leq r$, and we can use Lemma \ref{lem34} to assert that we have 
\begin{equation}\label{eq310}
\forall t \in [0,T], \; \Vert f(t,x_0(t), u_a(t)) - f(t, x(t), u_a(t)) \Vert \leq L \Vert x(t) - x_0(t) \Vert.
\end{equation}
Now for all $t \in [0,T]$ we have\\
$\Vert ( \Phi_a(x) - \Phi_a(x_0))(t) \Vert = \Vert \int_0^t ( f(s,x(s), u_a(s)) - f(s, x_0(s), u_a(s))) ds \Vert$ \\
$ \leq \int_0^t \Vert f(s, x(s), u_a(s)) - f(s, x_0(s), u_a(s)) \Vert ds \;\;\;\; \Longrightarrow $ \\
$e^{-L \cdot t} \Vert ( \Phi_a(x) - \Phi_a(x_0))(t) \Vert \leq e^{-L \cdot t} \int_0^t \Vert f(s, x(s), u_a(s)) - f(s, x_0(s), u_a(s)) \Vert ds$ \\
$ \leq e^{-L \cdot t} \int_0^t (L \Vert x(t) - x_0(t) \Vert) ds$ (after (\ref{eq310})) \\
$= L  e^{-L \cdot t} \int_0^t (e^{L \cdot s} e^{- L \cdot s} \Vert x(s) - x_0(s) \Vert ) ds \leq L  e^{-L \cdot t} \int_0^t (e^{L \cdot s} \Vert x - x_0 \Vert_b) ds$ \\
$ = L e^{-L \cdot t} \frac{e^{L \cdot t} - 1}{L} \Vert x - x_0 \Vert_b = (1 - e^{L \cdot t})\Vert x - x_0 \Vert_b \leq (1 - e^{- L \cdot T} )r_1$.\\
Taking the sup on the $t \in [0,T]$, we have proven
\begin{equation}\label{eq311}
\forall  x \in \mathcal{X}, \;\; \Vert \Phi_a(x) - \Phi_a(x_0) \Vert_b \leq (1 - e^{ - L \cdot T}) r_1.
\end{equation}
Using (\ref{eq39}), we obtain $\Vert \Phi_a(x) - \Phi_a(x_0) \Vert_b \leq $ \\
$\Vert \Phi_a(x) - \Phi_a(x_0) \Vert_b + \Vert \Phi_a(x_0) - x_0 \Vert_b \leq (1 - e^{-L \cdot T}) r_1 + e^{-L \cdot T} r_1 = r_1$, hence $\Phi_a(x) \in \mathcal{X}$.
\end{proof}
\begin{lemma}\label{lem36} The constant $r_2$ comes from Lemma \ref{lem35}.
Let $a \in \R^N_+$. If $\Vert a \Vert \leq r_2$, then, for all $x, z \in \mathcal{X}$, we have 
$\Vert \Phi_a(x) - \Phi_a(z) \Vert_b \leq (1 - e^{-L \cdot T}) \Vert x - z \Vert_b$.
\end{lemma}
\begin{proof}
Let $x, z \in \mathcal{X}$. Since, for all $t \in [0,T]$, we have $e^{- L \cdot t} \Vert x(t) - x_0(t) \Vert \leq r_1$ and $e^{- L \cdot t} \Vert z(t) - x_0(t) \Vert \leq r_1$, we obtain $\Vert x(t) - x_0(t) \vert \leq r$ and $\Vert z(t) - x_0(t) \Vert`\leq r$, and using Lemma \ref{lem34}, we have \\
$e^{-L \cdot t} \Vert( \Phi_a(x) - \Phi_a(z))(t) \Vert \leq e^{-L \cdot t}\int_0^t \Vert f(s,x(s),u_a(s)) - f(s,z(s),u_a(s)) \Vert ds$ \\
$\leq e^{-L \cdot t} \int_0^t(L \Vert x(s) - z(s) \Vert) ds 
= L e^{-L \cdot t} \int_0^t (e^{L \cdot s} e^{-L \cdot s} \Vert x(s) - z(s) \Vert) ds$ \\
$ \leq  L e^{-L \cdot t} \int_0^t (e^{L \cdot s} \Vert x - z \Vert_b)ds
\leq L e^{-L \cdot t} \frac{e^{L \cdot t} - 1}{L} \Vert x - z \Vert_b \leq (1 - e^{- L \cdot T}) \Vert x - z \Vert_b$.
\end{proof}
\begin{lemma}\label{lem37}
For all $x \in \mathcal{X}$, the mapping $[a \mapsto \Phi_a(x)]$ is continuous from $\overline{B}(0,r_2) \cap \R^N_+$ into $\mathcal{X}$.
\end{lemma}
\begin{proof}
Lemma \ref{lem33} ensures the continuity of this mapping at $a =0$. Now we fix $\hat{a} \neq 0$. Let $(a^n)_{n \in \N}$ be a sequence  in $\overline{B}(0,r_2) \cap \R^N_+$ which converges toward $\hat{a}$. Note that we have
\begin{equation}\label{eq312}
\left.
\begin{array}{cl}
f(t,x(t),u_a(t)) = & 1_{[0,t_1)} f(t,x_0(t),u_0(t)) +\\
\null &  \sum_{i=1}^N 1_{[t_i + b_i(a), t_i + b_i(a) + a_i)} f(t, x(t), v_i) +\\
\null & \sum_{i=1}^{N -1} 1_{[t_i + b_i(a) + a_i, t_{i+1} + b_{i+1}(a))} f(t, x(t), u_0(t))\\
\null & + 1_{[t_N + b_N(a) + a_N, T]} f(t, x_0(t), u_0(t)).
\end{array}
\right\}
\end{equation}
We denote by $\mu$ the positive measure of Borel-Lebesgue of $[0,T]$. We have\\
$\lim_{n \rightarrow + \infty} 1_{[t_i + b_i(a^n), t_i + b_i(a^n) + a_i^n)} (t) =  1_{[t_i + b_i(\hat{a}), t_i + b_i(\hat{a}) + \hat{a}_i)}(t)$, $\mu$-a.e. $t \in [0,T]$ since the pointwise convergence is clear when $t \in (t_i + b_i(\hat{a}), t_i + b_i(\hat{a}) + \hat{a}_i)$ and when $t \in [0,T] \setminus [t_i + b_i(\hat{a}), t_i + b_i(\hat{a}) + \hat{a}_i]$, and a finite set is a $\mu$-null set. Similarly we obtain
$\lim_{n \rightarrow + \infty}1_{[t_i + b_i(a^n) + a_i^n, t_{i+1} + b_{i+1}(a^n))}(t) = 1_{[ t_i + b_i(\hat{a}) + \hat{a}_i, t_{i+1} + b_{i+1}(\hat{a}) )}(t)$, $\mu$-a.e. $t \in [0,T]$, $\lim_{n \rightarrow + \infty}
 1_{[t_N + b_N(a^n) + a^n_N, T]}(t) =  1_{[t_N + b_N(\hat{a}) + \hat{a}_N, T]}(t)$, $\mu$-a.e. $t \in [0,T]$. Since a finite union of $\mu$-null sets is a $\mu$-null set, using (\ref{eq312})  we obtain
\begin{equation}\label{eq313}
\lim_{n \rightarrow + \infty} \Vert f(t, x(t), u_{a^n}(t)) - f(t, x(t), u_{\hat{a}}(t)) \Vert = 0, \; \; \mu-a.e.\; \; t \in [0,T].
\end{equation}
Let $( \tau_j)_{0 \leq j \leq k+1}$ be a subdivision of $[0,T]$ such that the discontinuity points of $u_0$ belong to $\{ \tau_j : 0 \leq j \leq k \}$. When $j \in \{ 0, ..., k \}$ we use the function $u_0^j$ defined in (\ref{eq33}) and then $\{f(t, x(t), u_0^j(t)) : t \in [ \tau_j, \tau_{j+1}] \}$ is compact an an image of a compact set by a continuous function. Since $\{ f(t, x(t), u_0(t)) : t \in [0,T] \}$ is included in the finite union of compact sets $\bigcup_{0 \leq j \leq k} \{f(t, x(t), u_0^j(t)) : t \in [ \tau_j, \tau_{j+1}] \}$, it is bounded.
  For all $i \in \{ 1, ..., N \}$, the set $\{ f(t, x(t), v_i) : t \in [0,T] \}$ is compact under (A4). Note that $\{ f(t, x(t), u_a(t)) : t \in [0,T], a \in \overline{B}(0,r_2) \cap \R_+^N \}$  is included in $ \{ f(t, x(t), u_0(t)) : t \in [0,T] \} \cup (\cup_{1 \leq i \leq N} \{ f(t,x(t), v_i) : t \in [0,T] \})$. This last set is bounded as a finite union of bounded sets, hence there exists $\sigma \in \R_{+ *}$ such that, for all $t \in [0,T]$ and for all $a \in \overline{B}(0,r_2) \cap \R_+^N $, $\Vert f(t, x(t), u_a(t)) \Vert \leq \frac{\sigma}{2}$. Hence we have 
\begin{equation}\label{eq314}
\exists \sigma \in \R_{+*}, \forall n \in \N, \forall t \in [0,T], \Vert f(t,x(t), u_{a^n}(t)) - f(t, x(t), u_{\hat{a}}(t)) \Vert \leq \sigma.
\end{equation}
Note that the constant $\sigma$ is $\mu$-integrable on $|0,T]$, and that the functions $[t \mapsto \Vert f(t,x(t), u_{a^n}(t)) - f(t, x(t), u_{\hat{a}}(t)) \Vert ]$ is a Borel function on $[0,T]$ as a composition of Borel functions. Hence, using (\ref{eq313}) and (\ref{eq314}), we can use the theorem of the dominated convergence of Lebesgue and assert that we have 
\begin{equation}\label{eq315} 
\lim_{n \rightarrow + \infty} \int_0^T \Vert f(t, x(t), u_{a^n}(t)) - f(t, x(t), u_{\hat{a}}(t)) \Vert dt = 0.
\end{equation}
For all $n \in \N$, for all $t \in [0,T]$, we have \\
$e^{-L \cdot t} \Vert (\Phi_{a^n}(x) - \Phi_{\hat{a}}(x))(t) \Vert \leq e^{-L \cdot t} \int_0^t \Vert f(s, x(s), u_{a^n}(s)) - f(s, x(s), u_{\hat{a}}(s))\Vert ds$ \\
$\leq \int_0^T \Vert f(t, x(t), u_{a^n}(t)) - f(t, x(t), u_{\hat{a}}(t)) \Vert dt$, then  taking the sup on the $t \in [0,T]$, and using (\ref{eq315}), we obtain $\lim_{n \rightarrow + \infty} \Vert \Phi_{a^n}(x) - \Phi_{\hat{a}}(x) \Vert_b = 0$.
\end{proof}
\begin{proposition}\label{prop38}
The following assertions hold.
\begin{itemize}
\item[(i)] For all $a \in \overline{B}(0,r_2) \cap \R^N_+$, there exists a solution $x_a$ of the Cauchy problem (\ref{eq32}) which is defined on $[0,T]$ all over.
\item[(ii)] The mapping $[a \mapsto x_a]$, from $\overline{B}(0,r_2) \cap \R^N_+$  into $\mathcal{X}$, is continuous.
\end{itemize}
\end{proposition}
\begin{proof}
From Lemma \ref{lem35}, Lemma \ref{lem36} and Lemma \ref{lem37} we can use Theorem \ref{th32} and assert that, for each $a \in \overline{B}(0,r_2) \cap \R^N_+$, there exists a unique fixed point $x_a$ of $\Phi_a$ in $\mathcal{X}$, and moreover we know that the mapping $[a \mapsto x_a]$ is continuous. From the definition (\ref{eq35}), we have $x_a(t) = \xi_0 + \int_0^t f(s, x_a(s), u_a(s)) ds$ for all $t \in [0,T]$. From (A4), we can see that the function $[s \mapsto f(s,x_a(s), u_a(s))]$ belongs to $NPC^0([0,T], E)$, and consequently the function $[t \mapsto \int_0^t f(s, x_a(s), u_a(s)) ds ]$ belongs to $PC^1([0,T],E)$, and using a classical result on the differentiation of the primitives functions (\cite{Bo}, chapter 2, Corollary 1, FVR. II6), we obtain that $\underline{d}x_a$ is well defined on $[0,T]$ and we have $\underline{d}x_a(t) = f(t, x_a(t), u_a(t))$ on $[0,T]$. We also have $x_a(0) = \xi_0$, and so $x_a$ is a solution of the Cauchy problem (\ref{eq32}). Hence the assertion (i) is proven, and the assertion (ii) results from the continuity of the fixed point with respect to $a$.
\end{proof}
\subsection{Properties of differentiability} In this subsection we establish the Fr\'echet differentiability of the mapping $[a \mapsto x_a(T)]$ at the origine.
\vskip1mm
\noindent
First we recall some properties of the resolvents. 
We consider the linear ODE $\underline{d}y(t) = D_2f(t, x_0(t), u_0(t)) y(t)$ when $t \in [0,T]$. Following the indications which are given in \cite{PB} (Chapter 18) we can assert that, denoting by $R(t,s)$ the resolvent of this linear equation, we have $R(t_3,t_1) = R(t_3, t_2)R(t_2, t_1)$, $R(s,s) = id_E$, $R(s,t) = R(t,s)^{-1}$, $R( \cdot, s) \in PC^1([0,T], \mathcal{L}(E,E))$. We define $\underline{d}_1 R(t,s) := \underline{d} R( \cdot, s)(t)$ and we have, for all $t \in [0,T]$, $\underline{d}_1 R(t,s) = D_2 f(t, x_0(t), u_0(t)) R(t,s)$, and from $R(t,s) = R(s,t)^{-1}$, we obtain that $R(t, \cdot) \in PC^1([0,T], \mathcal{L}(E,E))$. We set $\underline{d}_2R(t,s) := \underline{d}R(t, \cdot)(s)$.
\vskip1mm
\noindent
The second step is the following fundamental result due to Michel.
\begin{lemma}\label{lem39} (\cite{Mi2} Lemma 1)  There exist $r_3 \in (0, r_2)$, $\Lambda \in \mathcal{L}(\R^N, E)$ and a mapping $\varrho : \overline{B}(0, r_3) \cap \R^N_+ \rightarrow E$ such that $\lim_{a \rightarrow 0} \varrho(a) = 0$, and such that, for all $a \in \overline{B}(0,r_3) \cap \R^N_+$, we have
$x_a(T) = x_0(T) + \Lambda a + \Vert a \Vert \varrho(a)$.\\
More precisely, $\Lambda a = \sum_{i=1}^N a_i R(T,t_i)[f(t_i, x_0(t_i), v_i) - f(t_i, x_0(t_i), u_0(t_i))]$. 
\end{lemma}
The following result proves that the mapping $[a \mapsto x_a(T)]$ is a restriction of a mapping (defined on a neighborhood of the origine  in $\R^N$) which is Fr\'echet differentiable at the origine.
\begin{proposition}\label{prop310}
The constant $r_3$ and the linear mapping $\Lambda$ are provided by Lemma \ref{lem39}. There exist $r_4 \in (0, r_3]$ and  a mapping $\kappa \in C^0(\overline{B}(0,r_4), \Omega)$ which is Fr\'echet differentiable at $a = 0$ and which satisfies, for all  $a \in  \overline{B}(0,r_3) \cap \R^N_+$, $\kappa(a) = x_a(T)$, and $D \kappa (0) = \Lambda$.
\end{proposition}
\begin{proof}
As a norm on $\R^N$ we choose the norm associated to the usual inner product. We denote by $\pi$ the best approximation projector from $\R^N$ on the closed convex cone $\R^N_+$, \cite{Au} (p. 18, Theorem 1). We know that $\pi$ is $1$-Lipschitzean. It is easy to verify that $\pi (\overline{B}(0,r_3)) \subset (\overline{B}(0,r_3) \cap \R^N_+)$. Using Proposition \ref{prop38} note that the mapping $\varrho $ is continuous on $\overline{B}(0,r_3) \cap \R^N_+$ since we have 
\[
\varrho(a) = \left\{
\begin{array}{ccl}
\frac{1}{\Vert a \Vert} (x_a(T) - x_0(T) -\Lambda a) & {\rm if} & a \neq 0\\
0 & {\rm if} & a =0.
\end{array}
\right.
\]
We set $\overline{\varrho} := \varrho \circ \pi \in C^0(\overline{B}(0, r_3), E)$. We define $\kappa : \overline{B}(0,r_3) \rightarrow E$ by setting $\kappa(a) := x_0(T) + \Lambda a + \Vert a \Vert \overline{\varrho}(a)$. Then $\kappa$ is continuous since $\Lambda$ and $\overline{\varrho}$ are continuous. We have also $\lim_{a \rightarrow 0} \overline{\varrho}(a) = \varrho(0) = 0$ which implies that $\kappa$ is Fr\'echet differentiable at $0$, and  that $D \kappa(0) = \Lambda $. 
Since $x_0(T) \in \Omega$ with $\Omega$ open, since $\lim_{a \rightarrow 0}(\Lambda a + \Vert a \Vert \overline{\varrho}(a)) = 0$, reducing $r_3$ to $r_4 \in (0, r_3]$ we can assert that $\kappa (\overline{B}(0, r_4)) \subset \Omega$.
\end{proof}
\section{Proof of the principle for the problem of Mayer}
We describe the general method. When we fix $S = ((t_i, v_i))_{1 \leq i \leq N} \in \S$, we reduce the initial dynamic problem of Mayer to a finite-dimensional static optimization problem where the unknow is the vector $a$ of the thicknessess of the needles. Using a multiplier rule on this static problem we obtain a list of multipliers which is dependent on $S$. This is the matter of the first subsection. \\
In the second subsection we prove that we can choose such a list of multipliers which is independent of $S \in \S$, and from this particular list we build the multipliers and the adjoint function of Theorem \ref{th12}.
\subsection{Reduction to the finite dimension}
We arbitrarily fix $S \in \S$. Since $(x_0,u_0)$ is optimal for $(\mathcal{M}')$, $0$ is a solution of the following finite-dimensional optimization problem
\[
(\mathcal{F}_S) :=
\left\{
\begin{array}{cl}
{\rm Maximize} & g^0(x_a(T))\\
{\rm subject} \:\: {\rm to} & a \in {B}(0, r_4) \cap \R^N_+ \\
\null & \forall \alpha = 1, ..., m, \; \; g^{\alpha}(x_a(T)) \geq 0\\
\null & \forall \beta = 1, ..., q, \; \; h^{\beta} (x_a(T)) = 0.
\end{array}
\right.
\]
Using the mapping $\kappa$ of Proposition \ref{prop310} and $(b^*_i)_{1 \leq i \leq N}$, the dual basis of the canonical basis of $\R^N$, $0$ is also solution of the following finite-dimensional optimization problem
\[
(\mathcal{F}_S^1) :=
\left\{
\begin{array}{cl}
{\rm Maximize} & g^0(\kappa (a)) \\
{\rm subject} \; \; {\rm to} & a \in B(0, r_4) \\
\null & \forall \alpha = 1, ..., m, \; \; g^{\alpha} ( \kappa (a)) \geq 0 \\
\null & \forall \beta = 1, ..., q, \; \; h^{\beta}( \kappa (a)) = 0 \\
\null & \forall i = 1, ..., N, \; \; b^*_i a \geq 0
\end{array}
\right.
\]
 since, when $a \in B(0, r_4)$ is admissible for $(\mathcal{F}_S^1)$ then necessarily we have $a \in B(0, r_4) \cap \R^N_+$. The interest to introduce $(\mathcal{F}_S^1)$ is that this problem enters into the setting of the multiplier rule of \cite{Bl} while it is not the case for $(\mathcal{F}_S)$.\\
 Note that Michel in \cite{Mi2} works on $(\mathcal{F}_S)$, not on $(\mathcal{F}_S^1)$. To do that, he uses a multiplier rule given in \cite{Mi1}, which concerns problems on a convex cone. 
\begin{lemma}\label{lem41} Let $S = ((t_i,v_i))_{1 \leq i \leq N} \in \S$. There exist $(\lambda_{\alpha})_{0 \leq \alpha \leq m} \in \R^{1 + m}$ and $(\mu_{\beta})_{1 \leq \beta \leq q} \in \R^q$ which satisfy the following conditions.
\begin{itemize}
\item[(a)] $(\lambda_{\alpha})_{0 \leq \alpha \leq m}$ and $(\mu_{\beta})_{1 \leq \beta \leq q}$ are not simulteanous equal to zero.
\item[(b)] $\forall \alpha = 0, ..., m$, \; \; $\lambda_{\alpha} \geq 0$.
\item[(c)] $\forall \alpha = 1,...,m$, \; \; $\lambda_{\alpha}g^{\alpha}(x_0(T)) = 0$.
\item[(d)] $\forall i = 1,...,N$, $p(t_i) [f(t_i, x_0(t_i), v_i) - f(t_i, x_0(t_i), u_0(t_i))] \leq 0$, where\\
$p(t) := (\sum_{\alpha = 0}^m \lambda_{\alpha} D g^{\alpha}(x_0(T)) + \sum_{\beta = 1}^q \mu_{\beta} Dh^{\beta}(x_0(T))) R(T,t)$, $R(t,s)$ being defined just before Lemma \ref{lem39}.
\end{itemize}
\end{lemma}
\begin{proof}
Using Proposition \ref{prop310}, (A1) and (A2), the assumptions of Theorem 3.2 in \cite{Bl} are fulfilled, and so we know that there exist $(\lambda_{\alpha})_{0 \leq \alpha \leq m} \in \R^{1 + m}$, $(\mu_{\beta})_{1 \leq \beta \leq q} \in \R^q$, and $(\nu_i)_{1 \leq i \leq N} \in \R^N$ such that the following conditions are fulfilled.
\begin{itemize}
\item[(i)] $(\lambda_{\alpha})_{0 \leq \alpha \leq m}$, $(\mu_{\beta})_{1 \leq \beta \leq q}$ and $(\nu_i)_{1 \leq i \leq N}$ are not simultaneously equal to zero.
\item[(ii)] $\forall \alpha = 0,..., m$, $\lambda_{\alpha} \geq 0$.
\item[(iii)] $\forall i = 1,...,N$, $\nu_i \geq 0$.
\item[(iv)] $\forall \alpha = 1,..., m$, $\lambda_{\alpha} g^{\alpha}(x_0(T)) = 0$.
\item[(v)] $\forall i = 1,...,N$, $\nu_i b^*_i0 = 0$.
\item[(vi)] $\sum_{\alpha = 0}^m \lambda_{\alpha} D g^{\alpha}(x_0(T))  D \kappa (0) + \sum_{\beta = 1}^q \mu_{\beta} Dh^{\beta}(x_0(T))  D \kappa (0) + \sum_{i = 1}^N \nu_i b^*_i = 0$.
\end{itemize}
To prove (a), we proceed by contradiction, we assume that $(\lambda_{\alpha})_{0 \leq \alpha \leq m}$ and $(\mu_{\beta})_{1 \leq \beta \leq q}$ are equal to zero. Hence, using (i), we have $(\nu_i)_{1 \leq i \leq N}$ different to zero. Using (vi) we obtain $\sum_{i = 1}^N \nu_i b^*_i = 0$, and since the $b^*_i$ are linearly independent we obtain that $(\nu_i)_{1 \leq i \leq N}$  is equal to zero: this is a contradiction. Consequently (a) is proven. Assertion (b) comes from (i) and  (c) comes from (iv). When $a \in \R^N_+$, using (iii), we have $\nu_i a_i \geq 0$, and from (vi) we obtain
$$\sum_{\alpha = 0}^m \lambda_{\alpha} D g^{\alpha}(x_0(T)) \circ D \kappa (0)a + \sum_{\beta = 1}^q \mu_{\beta} Dh^{\beta}(x_0(T)) \circ D \kappa (0)a + \sum_{i = 1}^N \nu_i a_i= 0$$
which implies the following relation, for all $a \in \R^N_+$,
\begin{equation}\label{eq41}
\sum_{\alpha = 0}^m \lambda_{\alpha} D g^{\alpha}(x_0(T)) \circ D \kappa (0)a + \sum_{\beta = 1}^q \mu_{\beta} Dh^{\beta}(x_0(T)) \circ D \kappa (0)a \leq 0.
\end{equation}
Since $D \kappa (0) a = \sum_{i=1}^N a_i R(T, t_i) [ f(t_i, x_0(t_i), v_i) - f (t_i, x_0(t_i), u_0(t_i))]$, the relation \ref{eq41}) is equivalent to 
$$\forall a \in \R^N_+, \;\; \sum_{i=1}^N a_i p(t_i) [ f(t_i, x_0(t_i), v_i) - f (t_i, x_0(t_i), u_0(t_i))] \leq 0,$$
which is equivalent to the conclusion (d).
\end{proof}
%
\subsection{End of the proof of Part (I)}
In this subsection we follow \cite{Mi2}. Since the set of the lists of multipliers is a cone, we can normalized them by adding the condition $\sum_{\alpha = 0}^m \vert \lambda_{\alpha} \vert + \sum_{\beta = 1}^q \vert \mu_{\beta} \vert = 1$. When $S \in \S$, we define $K(S)$ as the set of the $((\lambda_{\alpha})_{0 \leq \alpha \leq m}, (\mu_{\beta})_{1 \leq \beta \leq q})$ which verify the conclusions (a, b, c, d) of Lemma \ref{lem41} and the additional condition $\sum_{\alpha = 0}^m \vert \lambda_{\alpha} \vert + \sum_{\beta = 1}^q \vert \mu_{\beta} \vert = 1$. Denoting by $\Sigma(0,1)$ the unit sphere of $\R^{1+m+q}$, we have $K(S) \subset \Sigma (0,1)$, $K(S)$ is closed since it is defined by wide inequalities and equalities, If $(S^{\ell})_{1 \leq \ell \leq n} = ( (t_i^{\ell}, v_i^{\ell})_{1 \leq i \leq N^{\ell}})_{1 \leq {\ell} \leq n}$ is a finite family of elements of $\S$, then setting $N := \sum_{{\ell}  = 1}^n N^{\ell}$, we can build $0 < s_1 \leq s_2 \leq ... \leq s_N < T$ and $w_1, w_2,..., w_N \in U$ such that $\bar{S} = (s_j, w_j)_{1 \leq j \leq N} \in \S$ and such that, for all ${\ell} \in \{ 1,...,   n \}$, for all $i \in \{1,...,N^{\ell} \}$, there exists a unique $j \in \{1,...,N \}$ verifying $t^{\ell}_i = s_j$; and then we take $w_j := v^{\ell}_i$. Note that, for all ${\ell} \in \{1,...,n \}$, the values of $S^{\ell}$ belong to the values of $\bar{S}$. If $((\lambda_{\alpha})_{0 \leq \alpha \leq m}, (\mu_{\beta})_{ 1 \leq \beta \leq q}) \in K(\bar{S})$, the conclusions 
 (a, b, c, d) of Lemma \ref{lem41} are satisfied for the values of ${S}$, they are also satisfied for the values of $S^{\ell}$ for alll ${\ell} \in \{1,...,n \}$, which implies that $((\lambda_{\alpha})_{0 \leq \alpha \leq m}, (\mu_{\beta})_{ 1 \leq \beta \leq q}) \in  \bigcap_{1 \leq \ell \leq n} K(S^{\ell}) \neq \emptyset$. Hence, this last finite intersection is nonempty.\\
Since $\Sigma (0,1)$ is compact, the finite intersection property of the closed subsets of $\Sigma (0,1)$ implies that $\bigcap_{S \in \S} K(S) \neq \emptyset$, \cite{DG} (p. 154, Appendix). Now we choose an element $((\lambda_{\alpha})_{0 \leq \alpha \leq m}, (\mu_{\beta})_{1 \leq \beta \leq q})$ in $\bigcap_{S \in \S} K(S)$, and we consider $p$ defined in the conclusion (d) of Lemma \ref{lem41} for this chosen $((\lambda_{\alpha})_{0 \leq \alpha \leq m}, (\mu_{\beta})_{1 \leq \beta \leq q})$. After the building of the $K(S)$, we see that the conclusions (NN), (Si) and (S$\ell$) are proven.
\vskip1mm
\noindent
We take $t \in (0,T)$ and $v \in U$, and then we have $(t,v) \in \S$. Then the conclusion (d) of Lemma \ref{lem41} implies $p(t) [f(t,x_0(t),v) - f(t, x_0(t), u_0(t))] \leq 0$. Doing $t \rightarrow 0+$ and $t \rightarrow T-$, we obtain the inequality for all $t \in [0,T]$. Hence the conclusion (MP.M) is proven.
\vskip1mm
\noindent
Now we want to prove that $p$ is a solution of the adjoint equation. Using the differentiability of $R( \cdot, s)$ outside of a finite set, $R(t,s) = R(s,t)^{-1}$, the Fr\'echet differentiability of the inversion operator $\mathcal{I} : Isom(E,E) \rightarrow Isom(E,E)$, $\mathcal{I}(L) := L^{-1}$, and the chain rule we obtain the following formula.
\begin{equation}\label{eq42}
\underline{d}_2 R(T,t) = - R(T,t) D_2f(t, x_0(t), u_0(t)).
\end{equation}
Differentiating $p(t) = (\sum_{\alpha = 0}^m \lambda_{\alpha} D g^{\alpha}(x_0(T)) + \sum_{\beta = 1}^q \mu_{\beta} D h^{\beta}(x_à(T)))  R(T,t)$ with respect to $t$, we obtain\\
$\underline{d} p(t) = (\sum_{\alpha = 0}^m \lambda_{\alpha} D g^{\alpha}(x_0(T)) + \sum_{\beta = 1}^q \mu_{\beta} D h^{\beta}(x_à(T))) \underline{d}_2 R(T,t)$  and using (\ref{eq42}), we obtain\\
$\underline{d} p(t) = (\sum_{\alpha = 0}^m \lambda_{\alpha} D g^{\alpha}(x_0(T)) + \sum_{\beta = 1}^q \mu_{\beta} D h^{\beta}(x_à(T)))(- R(T,t) D_2f(t, x_0(t), u_0(t))$ \\
$ = - p(t) D_2f(t, x_0(t), u_0(t)) = -  D_2 H_M(t, x_0(t), u_0(t), p(t))$, and so $p$ satisfies (AE).
\vskip1mm
\noindent
From the equality $R(T,T) = id_E$ and from the formula which defines $p$ we see that the conclusion (TC) holds. To prove (CH.M) we need the following result.
\begin{lemma}\label{lem42}
Let $\phi \in C^0([0,T] \times U, \R)$ and $u \in NPC^0([0,T], U)$ such that $\phi(t, u(t)) = \max_{\zeta \in U} \phi(t, \zeta)$ for all $t \in [0,T]$. Then $\bar{\phi} := [ t \mapsto \phi(t, u(t))] \in C^0([0,T], \R)$.
\end{lemma}
\begin{proof}
Since $u$ is right continuous on $[0,T)$ and $\phi$ is continuous, $\bar{\phi}$ is reght continuous on $[0,T)$. Since $u$ is left continuous at $T$ and $\phi$ is continuous, we have $\bar{\phi}$ is left continuous at $T$. Now we ought to prove that $\bar{\phi}$ is left continuous on $(0,T)$. Let $t \in (0,T)$; for all $h \in (-t, 0)$, we have $\phi(t, u(t+h)) \leq \phi(t, u(t))$ and $\phi(t + h, u(t)) \leq \phi(t+h, u(t+h))$, and  doing 
$h \rightarrow 0-$, we obtain $\phi(t, u(t-)) \leq \phi(t, u(t))$ and $\phi(t, u(t)) \leq \phi(t, u(t-))$. Hence we have $\phi(t, u(t-)) = \phi(t, u(t))$, i.e. $\bar{\phi}(t-) = \bar{\phi}(t)$.
\end{proof}
If we set $\phi(t, \zeta) := H_M(t, x_0(t), \zeta, p(t))$, from (MP.M) we have $\bar{\phi} = \bar{H}_M$ and the conclusion (CH.M) is proven.
Hence Part (I) of Theorem \ref{th12} is completely proven for the problem of Mayer.
\subsection{Proof of part (II)}
We need of the following result.
\begin{lemma}\label{lem43} 
Let $\phi \in C^0([0,T] \times U, \R)$ such that, for all $(t, \zeta) \in [0,T] \times U$, the partial derivative with respect to the first variable $\partial_1 \phi(t, \zeta)$ exists, and $\partial_1 \phi$ is continuous on $[0,T] \times U$. Let $u \in NPC^0([0,T], U)$ such that $\bar{\phi}(t) := \phi(t, u(t)) = \max_{\zeta \in U} \phi(t, \zeta)$. Then the two following assertions hold.
\begin{itemize}
\item[(i)] When $t$ is a continuity point of $u$, then $\bar{\phi}$ is differentiable at $t$ and we have $\bar{\phi}'(t) = \partial_1 \phi(t, u(t))$.
\item[(ii)] $\bar{\phi} \in PC^1([0,T], \R)$.
\end{itemize}
\end{lemma}
\begin{proof}
From Lemma \ref{lem42} we know that $\bar{\phi} \in C^0([0,T], \R)$. Let $t$ be a continuity point of $u$. For all $h > 0$ small enough, we set $\Delta(h) := \bar{\phi}(t + h) - \bar{\phi}(t)$. We have $\phi(t+h, u(t)) - \phi(t, u(t)) \leq \phi(t+h, u(t+h)) - \phi(t,u(t)) = \Delta(h)$ and $\phi(t+h, u(t+h)) - \phi(t, u(t+h)) \geq \phi(t+h, u(t+h)) - \phi(t, u(t)) = \Delta (h)$. Using a classical theorem of Lagrange for the functions of one real variable (\cite{ATF}, p. 142), we know that there exist $\theta^h_1$ and $\theta^h_2$ in $(0,1)$ such that $\partial_1 \phi(t + \theta^h_1 h, u(t)) h \leq \Delta(h) \leq \partial_1 \phi(t + \theta^h_2 h, u(t + h)) h$ which implies $\partial_1 \phi(t + \theta^h_1 h, u(t)) \leq \frac{1}{h} \Delta(h) \leq \partial_1(t + \theta^h_2 h, u(t+ h))$, and doing $h \rightarrow 0+$ and using the continuity of $\partial_1 \phi$ and the continuity of $u$ at $t$, we obtain $ \lim_{h \rightarrow 0+} \frac{\Delta(h)}{h} = \partial_1 \phi(t, u(t) $. These last inequalities imply that the right derivative $\bar{\phi}'_{R}(t)$ exists and is equal to $\partial_1 \phi(t, u(t))$. Doing a similar reasonning, we obtain that the left derivative $\bar{\phi}'_{L}(t)$ exists and is equal to $\partial_1 \phi(t, u(t))$. Hence assertion (i) is proven. \\
Assertion (ii) is a consequence of assertion (i) using the continuity of $\partial_1 \phi$ and the normalized piecewise continuity of $u$.
\end{proof}
Setting $\phi(t, \zeta) := H_M(t, x_0(t), \zeta, p(t))$, we have $\bar{\phi} = \bar{H}_M$ and Part (II) is a corollary of Lemma \ref{lem43}.
\subsection{Proof of Part (III)}
We proceed by contradiction; if there exists $t_0 \in [0,T]$ such that $p(t_0) = 0$, since (AE) is linear, by using the uniqueness of the solution of the Cauchy problem ((AE), $p(t_0) = 0$), we obtain that $p(t) = 0$ for all $t \in [0,T]$, notably $p(T) = 0$. Hence using (TC), (Si) and (S${\ell}$), (QC, 0) implies that $(\forall \alpha = 0, ..., m, \; \lambda_{\alpha} = 0)$ and $(\forall \beta = 1, ..., q, \mu_{\beta} = 0)$ which is a contradiction with (NN). Hence Part (III) is proven.
%
\section{Proof of the principle for the problem of Bolza}
It is well known that we can transform a problem of Bolza into a problem of Mayer \cite{PB} (p. 393, Chapter 18). We realize such a transformation to deduce Theorem \ref{th11} from Theorem \ref{th12}. We introduce an additional state variable denoted by $\sigma$. We set $X := (\sigma, x) \in \R \times \Omega$ as a new state variable; we set $F(t, (\sigma,x), u) := (f^0(t,x,u), f(t,x,u))$ as the new vectorfield; we set $G^0(\sigma,x) := \sigma + g^0(x)$, $G^{\alpha}(\sigma,x) := g^{\alpha}(x)$ when $\alpha = 1, ...,m$, and we set $H^{\beta}(\sigma, x) := h^{\beta}(x)$ when $\beta = 1, ..., q$. We formulate te new following problem of Mayer:
\[ (\mathcal{MB}) 
\left\{
\begin{array}{cl}
{\rm Maximize} & G^0(X(T))\\
{\rm subject} \; \; {\rm to} & X \in PC^1([0,T], \R \times \Omega), u \in NPC^0([0,T], U)\\
\null & \underline{d}X(t) = F(t, X(t), u(t)), \; X(0) = (0, \xi_0)\\
\null & \forall \alpha = 1, ..., m, \; \; G^{\alpha}(X(T)) \geq 0 \\
\null & \forall \beta = 1,..., q, \; \; H^{\beta}(X(T)) = 0.
\end{array}
\right.
\]
\subsection{Proof of Part (I)}
We denote by $\varpi_1 : \R \times E \rightarrow \R$ and by $\varpi_2 : \R \times E \rightarrow E$ the two projections. 
\vskip1mm
\noindent
When $(x,u)$ is an admissible process for $(\mathcal{B})$, setting $\sigma(t) := \int_0^t f(s, x(s),u(s)) ds$, we see that $((\sigma,x), u)$ is an admissible process for $(\mathcal{MB})$ and we have $G^0((\sigma,x))(T) = \int_0^T f^0(t,x(t), u(t)) dt + g^0(x(T))$. Conversely when $(X,u)$ is an admissible process for $(\mathcal{MB})$, setting $x := \varpi_2 \circ X$, we see that $(x,u)$ is an admissible process for $(\mathcal{B})$, and setting $\sigma := \varpi_1 \circ X$, we have $\int_0^T f^0(t,x(t), u(t)) dt + g^0((x(T)) = \sigma (T) + g^0(x(T)) = G^0(X(T))$. Hence since $(x_0,u_0)$ is optimal for $(\mathcal{B})$, we obtain that $(X_0, u_0) = ((\sigma_0,x_0), u_0)$ is optimal for $(\mathcal{MB})$. The assumptions of Theorem \ref{th11} imply that the assumptions of Theorem \ref{th12} are fulfilled, where $(\mathcal{M})$ is replaced by $(\mathcal{MB})$. Hence there exist $(\Lambda_{\alpha})_{0 \leq \alpha \leq m} \in \R^{1+ m}$, $(M_{\beta})_{1 \leq \beta \leq q} \in \R^q$ and $P \in PC^1([0,T], (\R \times E)^*)$ such that the conclusions of Theorem \ref{th12} hold.
\vskip1mm
\noindent
When $P \in (\R \times E)^*$, we define $p_0 \in \R$ and $p \in E^*$ by setting $p_0 := P(1,0)$ and $p \xi := P(0, \xi)$ for all $\xi \in E$, and so we have $P(r, \xi) = p_0 r + p \xi$ for all $(r, \xi) \in \R \times E$. The Hamiltonian of $(\mathcal{MB})$ is $H_M(t, (\sigma,x), u, (p_0,p)) := (p_0,p) F(t, (\sigma, x), u) = p_0 f^0(t,x,u) + p f(t,x,u)$. The conclusions of Theorem \ref{th12} provide the following conditions.
\begin{itemize}
\item[(i)] $(\Lambda_{\alpha})_{0 \leq \alpha \leq m}$ and $(M_{\beta})_{1 \leq \beta \leq q}$ are not simulteanously equal to zero.
\item[(ii)] $\forall \alpha = 0, ..., m$, $\Lambda_{\alpha} \geq 0$.
\item[(iii)] $\forall \alpha = 1, ..., m$, $\Lambda_{\alpha} G^{\alpha}(X_0(T)) = 0$.
\item[(iv)] $P(T) = \sum_{\alpha = 0}^m \Lambda_{\alpha} DG^{\alpha}(X_0(T)) + \sum_{\beta = 1}^q M_{\beta} D H^{\beta}(X_0(T))$.
\item[(v)] $\underline{d}P(t) = - D_2H_M(t, X_0(t), u_0(t), P(t))$ for all $t \in |0,T]$.
\item[(vi)] $H_M(t, X_0(t), u_0(t), P(t)) \geq H_M(t, X_0(t), \zeta, P(t))$ for all $t \in [0,T]$ and for all $\zeta \in U$.
\item[(vii)] $[t \mapsto H_M(t, X_0(t), u_0(t), P(t))] \in PC^1([0,T], \R)$.
\end{itemize}
We set $\lambda_{\alpha} := \Lambda_{\alpha}$ for all $\alpha = 0, ..., m$, and $\mu_{\beta} := M_{\beta}$ for all $\beta = 1,...,q$.
Hence (i) and (ii) imply that (NN) and (Si) of Theorem \ref{th11} hold. From (iii) we obtain $\lambda_{\alpha} g^{\alpha}(x_0(T)) = 0$ for all $\alpha = 1, ..., m$, and so (S${\ell}$) of Theorem \ref{th11} holds.
\vskip1mm
\noindent
About the partial differentials, note that we have, for the partial differentials with respect to the first variable: $D_1 G^0(\sigma,x_0(T)) = id_{\R}$, $D_1 G^{\alpha}(\sigma,x_0(T)) = 0$ when $\alpha = 1, ...,m$, $D_1 H^{\beta}(\sigma,x_0(T)) = 0$ when $\beta = 1, ...,q$, and for the partial differentials with respect to the second variable: $D_2 G^0(\sigma,x_0(T)) = Dg^0(x_0(T))$, $D_2 G^{\alpha}(\sigma, x_0(T)) = D g^{\alpha}(x_0(T))$ when $\alpha = 1, ...,m$, and $D_2 H^{\beta}(\sigma, x_0(T))= D h^{\beta}(x(T))$ when $\beta = 1, ...,q$. Hence from (iv) we deduce the two following relations.
\begin{equation}\label{eq51}
p_0(T) = \lambda_0.
\end{equation}
\begin{equation} \label{eq52}
p(T) = \sum_{\alpha = 0}^m \lambda_{\alpha} Dg^{\alpha}(x_0(T)) + \sum_{\beta = 1}^q \mu_{\beta}  D h^{\beta}(x_0(T)).
\end{equation}
This last equatility is just the conclusion (TC) of Theorem \ref{th11}.
\vskip1mm
\noindent
From (v) we obtain that $\underline{d}p_0(t) = 0$ for all $t \in [0,T]$, and then using (\ref{eq51}) we have the following relation.
\begin{equation}\label{eq53}
\forall t \in [0,T], \; \; p_0(t) = \lambda_0.
\end{equation}
From (v) we also deduce that, for all $t \in [0,T]$, we have\\
$\underline{d} p(t) = \lambda_0 D_2 f^0(t, x_0(t), u_0(t)) + p(t) D_2 f(t, x_0(t), u_0(t))$ which is (AE.B) of Theorem \ref{th11}.
\vskip1mm
\noindent
From (vi) we deduce that, for all $t \in [0,T]$ and for all $\zeta \in U$, we have \\
$\lambda_0 f^0(t, x_0(t), u_0(t)) + p(t) f(t, x_0(t), u_0(t)) \geq \lambda_0 f^0(t, x_0(t), \zeta) + p(t) f(t, x_0(t), \zeta)$ which is the conclusion (MP.B) of Theorem \ref{th11}.\\
From (vii),  since $H_M(t, X_0(t), u_0(t), P(t)) = H_B(t, x_0(t), u_0(t), p(t), \lambda_0)$ we obtain (CH.B). 
Hence Part (I) of Theorem \ref{th11} is completely proven.
\subsection{Proof of Part (II)}
Using Part (II) of Theorem \ref{th12} on ($\mathcal{MB}$), the existence and the continuity of $\partial_1f^0$ and of $\partial_1f$ imply the existence and the continuity of $\partial_1F$. We obtain that  $[t \mapsto H_B(t, x_0(t), u_0(t), p(t), \lambda_0)  = H_M(t, X_0(t), u_0(t), P(t))] \in PC^1([0,T], \R)$, and when $t$ is a continuity point of $u_0$, we have $\bar{H}'_B(t) = \bar{H}'_M(t) = \lambda_0 \partial_1f^0(t, x_0(t), u_0(t)) + p(t) \partial_1 f(t, x_0(t), u_0(t))$. Hence Part (II) is proven.
\subsection{Proof of Part (III)} We procced by contradiction assuming that there exists $t_* \in [0,T]$ such $(\lambda_0, p(t_*)) = (0,0)$. Since $\lambda_0 = 0$, (AE.B) becomes an homogeneous linear equation, and using the uniqueness of the cauchy problem ((AE.B), $p(t_*) = 0$), we obtain that $p$ is equal to zero on $[0,T]$, notably we have $p(T) = 0$. Hence using (TC), (Si), (S${\ell}$), (QC, 1) implies that $(\forall \alpha = 1,...,m, \lambda_{\alpha} = 0)$ and $(\forall \beta = 1,...,q, \mu_{\beta} = 0)$. Since $\lambda_0 = 0$, we have $(\forall \alpha = 0,...,m, \lambda_{\alpha} = 0)$ and $(\forall \beta = 1,...,q, \mu_{\beta} = 0)$ which is a contradiction with (NN).
\end{document}